\documentclass{article}
\usepackage[parfill]{parskip}
\usepackage[margin=1in]{geometry}
\usepackage{comment}
\usepackage{amsmath,amssymb,amsthm,mathtools,microtype}
\usepackage{hyperref}
\DeclareUnicodeCharacter{0141}{\L{}}
\usepackage[nameinlink,noabbrev]{cleveref}
\usepackage{algorithm}
\usepackage[noend]{algpseudocode}
\usepackage{xcolor}
\usepackage{subcaption}

\newtheorem{theorem}{Theorem}[section]
\newtheorem{lemma}[theorem]{Lemma}
\newtheorem{proposition}[theorem]{Proposition}
\newtheorem{remark}[theorem]{Remark}
\newtheorem{assumption}[theorem]{Assumption}
\newtheorem{corollary}[theorem]{Corollary}

\newcommand{\R}{\mathbb{R}}
\newcommand{\tr}{\operatorname{tr}}
\newcommand{\dist}{\operatorname{dist}}
\newcommand{\argmin}{\operatorname*{argmin}}

\newcommand{\gap}{\Delta}
\newcommand{\pot}{\Psi}

\title{On the Complexity of the BFGS Method for Smooth Convex Optimization {without Strong Convexity}}
\author{Lijun Ding\thanks{Corresponding author. University of California San Diego, Department of Mathematics, La Jolla, CA, USA (\texttt{l2ding@ucsd.edu}).}
\and
Qiujiang Jin\thanks{Independent Researcher, New York City, USA (\texttt{qiujiangjin0@gmail.com}).}
\and
Jinwen Yang\thanks{University of Chicago, Department of Statistics, Chicago, IL, USA (\texttt{jinweny@uchicago.edu}).}
\and
Baoyu Zhou\thanks{Arizona State University, School of Computing and Augmented Intelligence, Tempe, AZ, USA (\texttt{baoyu.zhou@asu.edu}).}
}
\date{}

\begin{document}
\maketitle

\begin{abstract}
We study the BFGS method with an Armijo-Wolfe line search for minimizing convex functions with Lipschitz-continuous gradients that are bounded below, without assuming strong convexity. We establish a global iteration complexity bound of $\mathcal{O}(k^{-1/2})$ for the smallest gradient norm among the first $k$ iterates. When the initial sublevel set is bounded, we show that the function value gap converges at a rate of $\mathcal{O}(k^{-1})$. \textcolor{black}{We also provide a quadratic function example to show that our convergence bound on the gradient norm is sharp. Furthermore, under the global Polyak-Lojasiewicz (PL) condition, we improve the sublinear convergence rate to the non-asymptotic global linear convergence rate.} 
More precisely, these guarantees apply to arbitrary positive definite initial Hessian approximations, with explicit bounds that quantify the dependence on initialization and the Armijo-Wolfe parameters. Our analysis leverages the classical trace-log-determinant potential function of~\cite[Equation~(2.5)]{byrd1989tool}. A key observation is that a critical inequality underlying the trace-log-determinant  potential function remains valid without strong convexity.
\par\medskip\noindent\textbf{Keywords:} Quasi-Newton methods; BFGS; Smooth convex optimization; Iteration complexity; Polyak--Lojasiewicz inequality.
\par\smallskip\noindent\textbf{Mathematics Subject Classification (2020):} 90C30, 90C53.
\end{abstract}

\section{Introduction}

Broyden-Fletcher-Goldfarb-Shanno (BFGS), proposed in \cite{broyden1970convergence,fletcher1970new,goldfarb1970family}, is a popular quasi-Newton method due to its ability to adapt to local curvature while maintaining a relatively low per-iteration cost compared with Newton's method. 
For twice continuously differentiable convex functions without strong convexity, classical global convergence results for BFGS are primarily qualitative~\cite{byrd1987global,powell1976algorithm} and do not provide explicit global convergence rates.
In the strongly convex setting, classical results establish asymptotic superlinear convergence~\cite{byrd1989tool,nocedal2006numerical}, while explicit global and local convergence rates for BFGS with suitable line searches have become available more recently; see, e.g.,
\cite{rodomanov2022rates,jin2023non,jin2024non,jin2025non,jin2026affine,jiang2023online}. 
In contrast, global convergence rates for standard BFGS on smooth convex functions without strong convexity remain largely unexplored.

This note presents iteration complexity guarantees for BFGS with an Armijo-Wolfe line search when minimizing a \emph{convex}, \emph{$L$-smooth} (i.e., the gradient is $L$-Lipschitz continuous), and \emph{bounded-below} function $f:\mathbb{R}^d \rightarrow \mathbb{R}$, starting from an arbitrary initial point.  
Specifically, for any iteration $k$, we show that the iterates $\{x_j\}_{0\leq j\leq k}$ generated by BFGS satisfy the best-iterate gradient bound: $\min_{0\leq j\leq k} \|\nabla f(x_j)\| \leq \mathcal{O}(k^{-1/2})$. 
Under the additional assumption that the distance between any point of the initial sublevel set and the optimal solution set is uniformly bounded, i.e., $\sup_{x\in\mathbb{R}^d:f(x)\leq f(x_0)}\dist(x,X^\star) <\infty$ with  $X^\star:=\argmin_{x\in\mathbb{R}^d} f(x)$, we further establish the convergence of function values satisfying $f(x_k) - f^\star \leq \mathcal{O}(k^{-1})$, where $f^\star := \min_{x\in\mathbb{R}^d}f(x)$.
\textcolor{black}{We also provide a quadratic function example to demonstrate the sharpness of our convergence bound on the gradient norm. Furthermore, under the Polyak–Łojasiewicz (PL) condition, we strengthen the sublinear convergence guarantee to an explicit non-asymptotic global linear convergence rate.}
Our proof builds on the classical trace-log-determinant potential function in~\cite[Equation~(2.5)]{byrd1989tool}. A key observation, formalized in~Lemma \ref{lem:potential}, is that a critical inequality associated with this potential function does not require strong convexity. This allows the classical analysis to be extended to the smooth convex setting and yields the stated global complexity bounds \textcolor{black}{without strong convexity}.

\section{Algorithm setup and key lemmas}

We consider using the BFGS method to solve optimization problems of the form
\begin{equation}\label{eq:prob}
\min_{x\in\mathbb{R}^d}\ f(x),
\end{equation}
where $f:\R^d\to\R$ is continuously differentiable and convex, and $\nabla f:\R^d\to\R^d$ is $L$-Lipschitz continuous with constant $L>0$.
The BFGS algorithm is described as follows.
\begin{algorithm}
\caption{BFGS Method with an Armijo-Wolfe Line Search}
\begin{algorithmic}[1]
\Require Initial iterate $x_0 \in \mathbb{R}^d$; initial positive definite Hessian approximation matrix $B_0\succ 0$; user-defined parameters $\{c_1,c_2\}$ satisfying $0 < c_1 < c_2 < 1$ and $0 < c_1 < \frac{1}{2}$ 
\For{$k=0,1,2,\ldots$}
\State Evaluate gradient $\nabla f(x_k)$
\State \textbf{if} $\|\nabla f(x_k)\| = 0$, \textbf{then} terminate
\State Compute search direction $p_k = -B_k^{-1}\nabla f(x_k)$
\State Identify a suitable step size $\alpha_k > 0$ satisfying the Armijo-Wolfe line-search condition:
\begin{align}
 & f(x_k + \alpha_k p_k) \le f(x_k)+c_1\alpha_k\nabla f(x_k)^\top p_k,   \label{eq: amijo_old} \tag{A} \\
\text{and} \quad & \nabla f(x_k + \alpha_kp_k)^\top p_k \ge c_2\nabla f(x_k)^\top p_k. \label{eq: wolfe_old}  \tag{W}
\end{align}
\State Update iterate $x_{k+1} = x_k + \alpha_k p_k$
\State Evaluate iterate displacement $s_k = x_{k+1} - x_k$ and gradient displacement $y_k = \nabla f(x_{k+1}) - \nabla f(x_k)$
\State Update Hessian approximation matrix 
\begin{equation*}
B_{k+1} = B_k-\frac{B_ks_ks_k^\top B_k}{s_k^\top B_ks_k}+\frac{y_ky_k^\top}{y_k^\top s_k}
\end{equation*}
\EndFor
\end{algorithmic}
\label{algo}
\end{algorithm}

For any iteration $k$, let $g_k:=\nabla f(x_k)$. By the use of $x_{k+1} = x_k + s_k = x_k + \alpha_kp_k$, the Armijo-Wolfe conditions described in Algorithm \ref{algo} are equivalent to
\begin{align}
& f(x_{k+1}) \le f(x_k)+c_1g_k^\top s_k,   \label{eq: amijo} \tag{A}\\
\text{and} \quad & g_{k+1}^\top s_k \ge c_2g_k^\top s_k. \label{eq: wolfe}  \tag{W}
\end{align}
Let $f^\star:=\inf_{x\in\mathbb{R}^d}f(x)$ denote the optimal function value of~\eqref{eq:prob}. By defining $\Delta_k := f(x_k) - f^\star$ as the optimal function value gap, we have the following technical result.

\begin{lemma}[Standard one-step estimates]\label{lem:onestep}
Before termination, it holds for any $k\geq 0$ that $B_k\succ0$, $g_k^\top s_k<0$,
\begin{align}
 y_k^\top s_k&\ge(1-c_2)(-g_k^\top s_k)>0,                 \label{eq:curv}\\
 \text{and} \quad \|y_k\|^2&\le L y_k^\top s_k.                            \label{eq:coco}
\end{align}
Consequently, with $\gamma:=c_1(1-c_2)$ and
\begin{equation}
 r_k:=\frac{L(g_k^\top s_k)^2}
 {\|g_k\|^2y_k^\top s_k},                                \label{eq:rdef}
\end{equation}
one has
\begin{equation}
 \gap_k-\gap_{k+1}\ge \frac{\gamma}{L}r_k\|g_k\|^2.
                                                                    \label{eq:descent}
\end{equation}
\end{lemma}

\begin{proof}
Positive definiteness of $B_k$ makes $p_k$ a descent direction.  The Wolfe condition \eqref{eq: wolfe} gives
\eqref{eq:curv}, and the BFGS update then preserves positive definiteness of $B_k^{-1}$ and $B_k$
\cite[Sec.~6.1, p. 141]{nocedal2006numerical}.  Equation \eqref{eq:coco} is a quick consequence of convexity and $L$-smoothness \cite[Theorem~2.1.5, Eq.~(2.1.11)]{Nesterov2018}; the
line-search ratios in \eqref{eq:rdef} also appear as $\frac{\cos(\hat{\theta}_k)^2}{\hat{m}_k}$ in \cite[Remark 1 and Proposition 2]{jin2025non}.
Finally,
equations \eqref{eq:curv} and \eqref{eq:rdef} give
$-g_k^\top s_k\ge(1-c_2)r_k\|g_k\|^2/L$, and the Armijo condition \eqref{eq: amijo} gives
\eqref{eq:descent}.
\end{proof}
\begin{remark} 
Note that convexity is only used in proving  \eqref{eq:coco}. The inequality is used in the following trace-log-determinant argument to control the sum $\sum_{i=1}^k (\frac{\|y_k\|^2}{L y_k^\top s_k}-1)$. Without convexity, this term can blow up \cite[Equation (2.6) and the two equations before (3.9)]{dai2013perfect}. 
\end{remark}

Define the classical BFGS trace-log-determinant potential function with a scale $L$ and shift $d$ for any symmetric positive definite $B \in \mathbb{R}^{d\times d}$:
\begin{equation}
 \pot_L(B):=\tr(B/L)-d-\log\det(B/L) \overset{(a)}{\ge} 0.                  \label{eq:potential}
\end{equation}
The inequality $(a)$ is not difficult to establish and we omit the details; see also \cite[Exercises 6.8 and 6.9]{nocedal2006numerical} and \cite[discussion after Eq. (18)]{jin2025non}.
The following result appears as a quick consequence of 
\cite[Eq.~(6.50)]{nocedal2006numerical}; see also
\cite{byrd1989tool}. A critical observation is that the following result \emph{does not} require strong convexity. We include a short derivation here for self-completeness.

\begin{lemma}[BFGS potential inequality]\label{lem:potential}
For every $k\ge1$ before termination,
\begin{equation}
 \sum_{i=0}^{k-1}\log r_i
 \ge-\pot_L(B_0)+\sum_{i=0}^{k-1}
 \left(1-\frac{\|y_i\|^2}{L y_i^\top s_i}\right)
 \ge-\pot_L(B_0).                                        \label{eq:product}
\end{equation}
In particular, $\prod_{i<k}r_i\ge e^{-\pot_L(B_0)}$.
\end{lemma}

\begin{proof}

For one BFGS step, suppress the index and define 
\[
 u:=\frac{\|y\|^2}{Ly^\top s},\qquad
 b:=\frac{\|Bs\|^2}{Ls^\top Bs},\qquad \text{and} \qquad
 \zeta:=\frac{y^\top s}{s^\top Bs}.
\]
Also denote $B_+ = B_{k+1}$. Using properties of the trace and determinant, we have:
\begin{equation}
 \pot_L(B_+)-\pot_L(B)=u-b-\log\zeta.                     \label{eq:potrec}
\end{equation}
Indeed, the trace change is $u-b$, while
$\det(B_+)/\det(B)=\zeta$; see also \cite[Exercises 6.10 and 6.11]{nocedal2006numerical}.  Since $Bs=-\alpha g$,
\[
 b\zeta=
 \frac{\|Bs\|^2y^\top s}{L(s^\top Bs)^2}
 =\frac{\|g\|^2y^\top s}{L(g^\top s)^2}=\frac1r.
\]
Replacing $\zeta$ in \eqref{eq:potrec} by $\frac{1}{rb}$, followed by
$\log b-b\le-1$, yields
\[
 \pot_L(B_+)-\pot_L(B)
 =u-b+\log b+\log r\le u-1+\log r.
\]
Summing the above from $i=0$ to $i= k-1$ and using $u\le1$ from Equation \eqref{eq:coco} and
$\pot_L(B_k)\ge0$, we prove Equation \eqref{eq:product}.
\end{proof}

\section{Best-gradient and function-gap complexity}
Now we are ready to show the iteration complexity result for global convergence of the BFGS algorithm.
\begin{theorem}\label{thm:main}
Consider the BFGS method with Armijo-Wolfe line search for minimizing a convex, $L$-smooth function $f$ that is bounded below. Let $\gamma:=c_1(1-c_2)$, $\pot_0=\pot_L(B_0)$ and denote the function gap $\Delta_k := f(x_k) - f^\star$ for any $k\geq 0$. For any $k$, if some iterate $x_j$, $0\le j\le k$, is optimal, then BFGS
terminates.  Otherwise,
\begin{equation}
 \min_{0\le i<k}\|g_i\|^2
 \le \frac{L\gap_0}{\gamma k}\exp\!\left(\frac{\pot_0}{k}\right).
                                                                    \label{eq:grad}
\end{equation}
Moreover, if $X^\star$, the optimal solution set of $f$, is non-empty and $W_0:=\sup_{\substack{x\notin X^\star:f(x)\le f(x_0)}}
 \frac{f(x)-f^\star}{\|\nabla f(x)\|}<\infty$, then
\begin{equation}
 \gap_k\le
 \left[\frac1{\gap_0}+\frac{\gamma k}{LW_0^2}
 e^{-\pot_0/k}\right]^{-1}.                              \label{eq:gap}
\end{equation}
In particular, if 
additionally $B_0=LI$, \eqref{eq:gap} implies 
\begin{equation}\label{eq: gk_Deltak_1_over_k_bound}
\frac{\|g_k\|^2}{2L} \leq \Delta_k \leq \frac{LW_0^2}{\gamma k}\leq \frac{LR_0^2}{\gamma k}, 
\end{equation}
where $R_0 :\,= \sup_{x:f(x)\leq f(x_0)} \dist(x,X^*)$.
\end{theorem}

\begin{proof}
Summing \eqref{eq:descent} and using the Arithmetic Mean–Geometric Mean inequality (AM-GM) with
\eqref{eq:product} gives
\[
 \gap_0\ge\frac{\gamma}{L}\sum_{0}^{k-1}r_i\|g_i\|^2
 \ge\frac{\gamma k}{L}e^{-\pot_0/k}\min_{0\leq i<k}\|g_i\|^2,
\]
which proves \eqref{eq:grad}. 

Let $w_i:=\gap_i/\|g_i\|$ and
$\bar w_k=(\prod_{i<k}w_i)^{1/k}$. From \eqref{eq:descent} in Lemma \ref{lem:onestep} and the fact that $\Delta_i \ge \Delta_{i+1}$ due to the Armijo condition \eqref{eq: amijo}, we have 
\[
 \frac1{\gap_{i+1}}-\frac1{\gap_i}
 =\frac{\gap_i-\gap_{i+1}}{\gap_i\gap_{i+1}} \ge 
 \frac{\gap_i-\gap_{i+1}}{\gap_i^2}
 \ge\frac{\gamma}{L}\frac{r_i}{w_i^2}.
\]
Telescoping and applying AM-GM once more proves
\begin{equation}
 \frac1{\gap_k}\ge\frac1{\gap_0}
 +\frac{\gamma k}{L(\prod_{i<k}w_i)^{2/k}}
   \exp\!\left(-\frac{\pot_0}{k}\right).                 \label{eq:gapapost}
\end{equation}
The definition of $W_0$ and the descent property of the Armijo condition \eqref{eq: amijo} imply that $w_i\le W_0$. Thus, we reach \eqref{eq:gap}. Finally, combining (i) $W_0\le R_0$ as convexity gives
$\gap_i\le\|g_i\|\dist(x_i,X^\star)$, (ii) the $L$-smoothness gives $\|g_k\|^2 \leq 2L \Delta_k$ \cite[Theorem 2.1.5, Eq. (2.1.10)]{Nesterov2018}, and (iii) $\Psi_0 = 0$ for $B = LI$, we reach \eqref{eq: gk_Deltak_1_over_k_bound}.
\end{proof}

\section{A matching best-gradient lower bound}
\label{sec:gradient-lower}

The following quadratic shows that the bound in~\eqref{eq:grad}
is sharp for bounds uniform in the dimension, for the fixed
line-search parameters specified below.

\begin{proposition}[Best-gradient lower bound]
\label{prop:gradient-lower}
Fix $0<c_1<c_2<1$ with $c_1\le 1/2$.
For every integer $T\ge1$ and $L,\Delta>0$, there exists a strongly
convex, globally $L$-smooth quadratic $f:\R^{T+1}\to\R$ with
$f(0)-f^\star=\Delta$ such that Algorithm~\ref{algo}, initialized
with $x_0=0$ and $B_0=I$, admits the Armijo--Wolfe steps
$\alpha_k=4/L$, $0\le k<T$, and satisfies
\begin{equation}
 \min_{0\le k\le T}\|\nabla f(x_k)\|
 =\sqrt{\frac{L\Delta}{2(T+1)}}.
 \label{eq:gradient-lower}
\end{equation}
These steps also 
are the unique
exact line-search steps.
\end{proposition}

\begin{proof}
Set $n=T+1$ and $a=\sqrt{L\Delta/(2n)}$. With the convention
$x_{n+1}=0$, define
\begin{equation}
 f(x)=\frac{L}{8}\sum_{j=1}^{n}
 \left(x_j-x_{j+1}-\frac{4a}{L}\right)^2.
 \label{eq:lower-quadratic}
\end{equation}
Its unique minimizer and minimum value are
\[
 x^\star=\frac{4a}{L}(n,n-1,\ldots,1)^\top,
 \qquad f^\star=0.
\]
Moreover, $f(0)=2na^2/L=\Delta$. For every nonzero $z\in\R^n$,
with $z_{n+1}=0$,
\[
 0<z^\top\nabla^2f\,z
 =\frac{L}{4}\sum_{j=1}^{n}(z_j-z_{j+1})^2
 \le L\|z\|^2.
\]
Thus $f$ is strictly convex and globally $L$-smooth.

Let $e_1,\ldots,e_n$ be the standard basis. Consider the partial
staircases
\begin{equation}
 x_k=\frac{4a}{L}\sum_{j=1}^{k}(k+1-j)e_j,
 \qquad 0\le k\le T.
 \label{eq:lower-staircase}
\end{equation}
At $x_k$, the first $k$ adjacent differences equal $4a/L$,
while the remaining differences vanish. Consequently,
\begin{equation}
 g_k=-ae_{k+1},
 \qquad
 f(x_k)=\frac{2(n-k)a^2}{L}.
 \label{eq:lower-values}
\end{equation}

We verify directly that these points satisfy the BFGS with $\alpha_k =\frac{4}{L}$.
Writing $v_j=\sum_{i=1}^{j}e_i$, the Hessian approximations satisfy
\begin{equation}
 B_k=\sum_{j=k+1}^{n}e_je_j^\top
 +\frac{L}{4}\sum_{j=1}^{k}
 (e_j-e_{j+1})(e_j-e_{j+1})^\top.
 \label{eq:lower-B}
\end{equation}
Indeed, this holds at $k=0$. If it holds at $k<T$, then
$B_k\succ0$ and $B_kv_{k+1}=e_{k+1}$, so
\[
 p_k=-B_k^{-1}g_k=av_{k+1},
 \qquad
 s_k=\frac{4a}{L}v_{k+1}=x_{k+1}-x_k.
\]
By~\eqref{eq:lower-values},
\[
 y_k=a(e_{k+1}-e_{k+2}),
 \qquad
 B_ks_k=\frac{4a}{L}e_{k+1},
 \qquad
 s_k^\top y_k=\frac{4a^2}{L}>0.
\]
Substitution into the BFGS update gives
\[
 B_{k+1}
 =B_k-e_{k+1}e_{k+1}^\top
 +\frac{L}{4}(e_{k+1}-e_{k+2})(e_{k+1}-e_{k+2})^\top,
\]
which proves~\eqref{eq:lower-B} at $k+1$.

Finally, for every $k<T$,
\[
 g_k^\top s_k=-\frac{4a^2}{L},
 \qquad
 f(x_k)-f(x_{k+1})=\frac{2a^2}{L},
 \qquad
 g_{k+1}^\top s_k=0.
\]
These identities verify~\eqref{eq: amijo}--\eqref{eq: wolfe}
because $c_1\le1/2$ and $c_2>0$.
The zero endpoint directional derivative and strict convexity
show that $\alpha_k=4/L$ is the unique exact line-search step.
Equation~\eqref{eq:lower-values} gives $\|g_k\|=a$ for every
$0\le k\le T$, proving~\eqref{eq:gradient-lower}.
\end{proof}


The constructed trajectory also coincides with that of the conjugate gradient method, in accordance with the classical
BFGS--conjugate-gradient relationship~\cite{nazareth1979relationship}, though the connection is not needed for the proof above.

\section{Global linear convergence under the PL condition}

In this section, we consider a convex function satisfying the Polyak-Lojasiewicz (PL) condition and show that the global sublinear convergence rate from Theorem~\ref{thm:main} improves to a nonasymptotic global linear convergence rate. First, we present the PL condition.

\begin{assumption}\label{ass_PL_condition}
The objective $f$ satisfies the Polyak-Lojasiewicz condition with parameter $\mu > 0$, i.e.,
\begin{equation*}
\frac{1}{2}\|\nabla{f(x)}\|^2 \geq \mu(f(x) - f^*), \quad \forall x \in \mathbb{R}^{d}.
\end{equation*}
\end{assumption}

In the following theorem, we improve the sublinear convergence rate from Theorem~\ref{thm:main} to the linear convergence rate under the PL condition.

\begin{theorem} \label{thm:linear_rate} Suppose the function is convex, $L$-smooth, and the PL condition in Assumption~\ref{ass_PL_condition} holds. Let $\{x_k\}_{k\geq 0}$ be the iterates generated by BFGS, where the step size satisfies the Armijo-Wolfe conditions. Denote the function value gap $f(x_k) - f(x^*)$. For any initial point $x_0 \in \mathbb{R}^{d}$ and any initial Hessian approximation matrix $B_0 \in \mathbb{S}^d_{++}$, we have the following global linear convergence rate (recall $\gamma= c_1(1-c_2)$):
\begin{equation}\label{eq:linear_rate}
    \Delta_k \leq \left(1 - 2\gamma e^{-\frac{\Psi_{L}(B_0)}{k}}\frac{\mu}{L}\right)^{k} {\Delta_0}, \quad \forall k \geq 1, 
\end{equation}
In particular, if the initial Hessian approximation matrix is set as $B_0 = LI$, then for any $k \geq 1$, we have that
\begin{equation}\label{eq:linear_rate_L}
    \frac{\Delta_k}{\Delta_0} \leq \left(1 - \frac{2\gamma\mu}{L}\right)^{k}.
\end{equation}
\end{theorem}

\begin{proof}
Recall from Lemma~\ref{lem:onestep}, we have that $\gap_k-\gap_{k+1}\ge \frac{\gamma}{L}r_k\|g_k\|^2$. From Assumption~\ref{ass_PL_condition}, we know that $\|g_k\|^2 \geq 2\mu \Delta_k$. Therefore, we obtain that
\begin{equation}
    \gap_k-\gap_{k+1}\ge \frac{2\gamma \mu}{L}r_k \Delta_k.
\end{equation}
This is equivalent to
\begin{equation}
    0 \leq \frac{\Delta_{k + 1}}{\Delta_k} \leq 1 - \frac{2\gamma \mu}{L}r_k.
\end{equation}
Therefore, we obtain that
\begin{equation}
\begin{split}
    & \frac{\Delta_k}{\Delta_0} = \prod_{i = 0}^{k - 1}\frac{\Delta_{i + 1}}{\Delta_i} \leq \prod_{i = 0}^{k - 1}( 1 - \frac{2\gamma \mu}{L}r_i) \leq \left(\frac{1}{k}\sum_{i = 0}^{k - 1}(1 - \frac{2\gamma \mu}{L}r_i)\right)^{k} \\
    & = \left(1 - \frac{2\gamma \mu}{L}\frac{1}{k}\sum_{i = 0}^{k - 1}r_i\right)^{k} \leq \left(1 - \frac{2\gamma \mu}{L} (\prod_{i = 0}^{k - 1}r_i)^{\frac{1}{k}}\right)^{k} \leq \left(1 - 2\gamma e^{-\frac{\Psi_{L}(B_0)}{k}}\frac{\mu}{L}\right)^{k}
\end{split}
\end{equation}
where we apply the AM-GM inequality twice in the second and third inequalities and the last inequality is due to results from Lemma~\ref{lem:potential}.

For $B_0 = LI$, note $\Psi_L(B_0) = 0$ and thus it achieves the inequality \eqref{eq:linear_rate_L}.
\end{proof}

Therefore, when the PL condition holds, the function  value gap $f(x_k) - f^*$ with linear convergence rate $\mathcal{O}(1 - \frac{\mu}{L})$. This holds even when the objective function is not strongly convex.

\section{Conclusion}

The BFGS potential yields an unconditional $\mathcal{O}(k^{-1/2})$ best-gradient guarantee and an $\mathcal{O}(k^{-1})$ function-gap rate under the gap-to-gradient bound $W_0<\infty$ or the boundedness of the distance of the initial sublevel set relative to the optimal solution set, i.e., $R_0<\infty.$ \textcolor{black}{Under the PL condition, the sublinear convergence rate can be enhanced to a global linear convergence rate of $\mathcal{O}(1 - \frac{\mu}{L})$.}

\section*{Acknowledgments} The authors recognize that the main proof ingredients appear in the literature, namely the one-step estimate and the trace-log potential.  In 2024, the second author had an unpublished proof (with a minor gap) of the gradient-norm bound, a weaker function-value gap bound, and the bound under the PL condition. We also thank Hang Yu for conducting preliminary numerical experiments and providing empirical results that support 
the convergence rates established by our theoretical analysis.

\bibliographystyle{plain}
\bibliography{references}

@article{byrd1989tool,
  title={A tool for the analysis of quasi-Newton methods with application to unconstrained minimization},
  author={Byrd, Richard H and Nocedal, Jorge},
  journal={SIAM Journal on Numerical Analysis},
  volume={26},
  number={3},
  pages={727--739},
  year={1989},
  publisher={SIAM}
}

@article{jin2025non,
  title={Non-asymptotic global convergence rates of {BFGS} with exact line search},
  author={Jin, Qiujiang and Jiang, Ruichen and Mokhtari, Aryan},
  journal={Mathematical Programming},
  pages={1--38},
  year={2025},
  publisher={Springer}
}

@article{jin2024non,
  title={Non-asymptotic global convergence analysis of {BFGS} with the {Armijo-Wolfe} line search},
  author={Jin, Qiujiang and Jiang, Ruichen and Mokhtari, Aryan},
  journal={Advances in Neural Information Processing Systems},
  volume={37},
  pages={16810--16851},
  year={2024}
}

@article{jin2026affine,
  title={Affine-invariant global non-asymptotic convergence analysis of bfgs under self-concordance},
  author={Jin, Qiujiang and Mokhtari, Aryan},
  journal={Advances in Neural Information Processing Systems},
  volume={38},
  pages={87920--87939},
  year={2025}
}

@book{Nesterov2018,
  author    = {Yurii Nesterov},
  title     = {Lectures on Convex Optimization},
  edition   = {Second},
  series    = {Springer Optimization and Its Applications},
  volume    = {137},
  publisher = {Springer},
  address   = {Cham},
  year      = {2018},
  doi       = {10.1007/978-3-319-91578-4}
}

@book{nocedal2006numerical,
  title={Numerical optimization},
  author={Nocedal, Jorge and Wright, Stephen J},
  year={2006},
  publisher={Springer}
}

@article{dai2013perfect,
  title={A perfect example for the BFGS method},
  author={Dai, Yu-Hong},
  journal={Mathematical Programming},
  volume={138},
  number={1},
  pages={501--530},
  year={2013},
  publisher={Springer}
}

@article{nazareth1979relationship,
  author  = {Nazareth, Larry},
  title   = {A Relationship between the {BFGS} and Conjugate Gradient
             Algorithms and Its Implications for New Algorithms},
  journal = {SIAM Journal on Numerical Analysis},
  volume  = {16},
  number  = {5},
  pages   = {794--800},
  year    = {1979},
  doi     = {10.1137/0716059}
}

@incollection{powell1976algorithm,
  author    = {M. J. D. Powell},
  title     = {Some Global Convergence Properties of a Variable Metric Algorithm for Minimization without Exact Line Searches},
  booktitle = {Nonlinear Programming},
  editor    = {Richard W. Cottle and Carlton E. Lemke},
  series    = {SIAM--AMS Proceedings},
  volume    = {9},
  pages     = {53--72},
  publisher = {American Mathematical Society},
  address   = {Providence, RI},
  year      = {1976}
}

@article{broyden1970convergence,
  title={The convergence of a class of double-rank minimization algorithms 1. General considerations},
  author={Broyden, Charles George},
  journal={IMA Journal of Applied Mathematics},
  volume={6},
  number={1},
  pages={76--90},
  year={1970},
  publisher={Oxford University Press}
}

@article{fletcher1970new,
  title={A new approach to variable metric algorithms},
  author={Fletcher, Roger},
  journal={The computer journal},
  volume={13},
  number={3},
  pages={317--322},
  year={1970},
  publisher={Oxford University Press}
}

@article{goldfarb1970family,
  title={A family of variable-metric methods derived by variational means},
  author={Goldfarb, Donald},
  journal={Mathematics of computation},
  volume={24},
  number={109},
  pages={23--26},
  year={1970}
}

@article{byrd1987global,
  title={Global convergence of a cass of quasi-Newton methods on convex problems},
  author={Byrd, Richard H and Nocedal, Jorge and Yuan, Ya-Xiang},
  journal={SIAM Journal on Numerical Analysis},
  volume={24},
  number={5},
  pages={1171--1190},
  year={1987},
  publisher={SIAM}
}

@article{rodomanov2022rates,
  title={Rates of superlinear convergence for classical quasi-Newton methods},
  author={Rodomanov, Anton and Nesterov, Yurii},
  journal={Mathematical Programming},
  volume={194},
  number={1},
  pages={159--190},
  year={2022},
  publisher={Springer}
}

@article{jin2023non,
  title={Non-asymptotic superlinear convergence of standard {quasi-Newton} methods},
  author={Jin, Qiujiang and Mokhtari, Aryan},
  journal={Mathematical Programming},
  volume={200},
  number={1},
  pages={425--473},
  year={2023},
  publisher={Springer}
}

@InProceedings{jiang2023online,
  author      = {Jiang, Ruichen and Jin, Qiujiang and Mokhtari, Aryan},
  booktitle   = {Proceedings of Thirty Sixth Conference on Learning Theory},
  title       = {Online Learning Guided Curvature Approximation: A Quasi-Newton Method with Global Non-Asymptotic Superlinear Convergence},
  pages       = {1962--1992},
  volume      = {195},
  year        = {2023},
}

\end{document}